\documentclass[a4paper]{article}
\pdfoutput=1
\usepackage[english]{babel}
\usepackage[utf8x]{inputenc}
\usepackage[T1]{fontenc}

\usepackage[a4paper,top=3cm,bottom=2cm,left=1cm,right=1cm,marginparwidth=1.75cm]{geometry}

\usepackage{amsmath}
\usepackage{amsfonts}
\usepackage{graphicx}
\usepackage{placeins}
\usepackage[colorinlistoftodos]{todonotes}
\usepackage[colorlinks=true, allcolors=blue]{hyperref}
\usepackage{makecell}

\usepackage{amsthm}
\theoremstyle{plain}
\newtheorem{theorem}{Theorem}[section]
\newtheorem{lemma}[theorem]{Lemma}

\newtheorem{conjecture}[theorem]{Conjecture}

\theoremstyle{remark}

\newtheorem{claim}{Claim}

\title{The Steklov Problem on Rectangles and Cuboids}
\author{Arnold Tan Junhan}
\date{}

\begin{document}
\maketitle

\begin{abstract}
\noindent This paper is a brief account of the Steklov eigenvalue problem on a 2-dimensional rectangular domain, and then on a 3-dimensional rectangular box. It is divided into four sections. Section 1 relies heavily on real analytic methods to show the existence of an eigenfunction class which always produces the first non-trivial Steklov eigenvalue on a rectangle. Section 2 lists all possible Steklov eigenfunctions on a cuboid. The very brief section 3 gives the 3-dimensional analogue of the analytic results in Section 1. The analogue is given as conjecture, but is expected to derivate from standard (albeit tedious) real analysis methods, should one wish to expound on these calculations. Section 4 deals with the special cases of the square and the cube.
\end{abstract}

\section{Analysis of the 1st Steklov eigenspace on the rectangle}
Consider the rectangle defined on $[-1,1]\times[-a,a]$, where $a \leq 1$. One can classify the Steklov eigenfunctions on this domain into 4 parity classes (Auchmuty $\&$ Cho, 2014). Generally, each parity class has two eigenfunctions, but for the case of the square $(a=1)$, there is an additional eigenfunction $s(x,y)=xy$ in Class II, whose eigenvalue is simple and equal to $1$.
\vspace{1mm}
\\
One can consider the smallest nontrivial eigenvalue $\sigma_1$, and ask: \textit{Among all rectangular domains with constant perimeter, which domains maximize the first Steklov eigenvalue, $\sigma_1$?}
\vspace{1mm}
\\
\noindent One may normalize the eigenvalue $\sigma_1$ by multiplying with the perimeter of the domain, $L=4(1+a)$. The product, $\sigma_1L$, is an invariant with respect to a scaling of the domain.
\vspace{1mm}
\\
\noindent That is, one can define an equivalence relation between similar rectangles, with rectangles in the same equivalence class having the same first Steklov invariant. We identify any rectangle with a rectangle defined $[-1,1] \times [-a,a]$, where $a \leq 1$, so that we may restrict our analysis to such rectangles. Since similar rectangles have the same Steklov invariants, the question then becomes: \textit{Among all values of $a \in (0,1]$, which maximizes $4\sigma_1(1+a)$?}
\vspace{1mm}
\\
\noindent We show that $a=1$ does.
\\ As mentioned above, we consider the eigenfunctions by parity class. In the table below, we only need to consider positive $\nu$.

\FloatBarrier
\begin{table}[!htbp]
\centering
\begin{tabular}{l|c|c|r}
Class & Eigenfunction & Determining Equation & Eigenvalue $\sigma$\\\hline\hline
I(i) & cosh$(\nu x)$cos$(\nu y)$, & tan$(\nu a)+$tanh$(\nu)=0$ & $\nu$tanh$(\nu)$ \\
I(ii) & cos$(\nu x)$cosh$(\nu y)$ & tan$(\nu)+$tanh$(\nu a)=0$ & $\nu$tanh$(\nu a)$ \\\hline
II(i) & sinh$(\nu x)$sin$(\nu y)$, & tan$(\nu a)=$tanh$(\nu)$ & $\nu$coth$(\nu)$ \\
II(ii) & sin$(\nu x)$sinh$(\nu y)$ & tan$(\nu)=$tanh$(\nu a)$ & $\nu$coth$(\nu a)$ \\\hline
III(i) & cosh$(\nu x)$sin$(\nu y)$, & tan$(\nu a)=$coth$(\nu)$ & $\nu$tanh$(\nu)$ \\
III(ii) & cos$(\nu x)$sinh$(\nu y)$ & tan$(\nu)+$coth$(\nu a)=0$ & $\nu$coth$(\nu a)$ \\\hline
IV(i) & sinh$(\nu x)$cos$(\nu y)$, & tan$(\nu a)+$coth$(\nu)=0$ & $\nu$coth$(\nu)$ \\
IV(ii) & sin$(\nu x)$cosh$(\nu y)$ & tan$(\nu)=$coth$(\nu a)$ & $\nu$tanh$(\nu a)$
\end{tabular}
\caption{\label{eightclasses}The eigenfunctions fall into eight classes.}
\end{table}
\FloatBarrier

\vspace{1mm}
\noindent (We have omitted $s(x,y)=xy$, which only arises when $a=1$.)

\vspace{2mm}
\noindent Then, for instance, the 1st Steklov eigenfunction from class IV(ii) has value $\nu$ given by the 1st positive solution of tan$\nu = $coth$(\nu a)$. Of course, for given $a$, minimizing $\nu$ is equivalent to minimizing $\nu$tanh$(\nu a)$. In fact, we have the following simple lemma. 

\begin{lemma}\label{increasing}
For given $a \in (0,1]$, the maps $\nu \mapsto \nu$tanh$(\nu a)$ and $\nu \mapsto \nu$coth$(\nu a)$ are increasing on $\nu \in (0,\infty)$ (Auchmuty $\&$ Cho, 2014).
\end{lemma}

\vspace{1mm}
\noindent We claim that the 1st nontrivial Steklov eigenvalue for a rectangle always comes from class IV(ii).

\vspace{1mm}
\noindent By considering intersections of graphs (see Figure 1 of Girouard $\&$ Polterovich, 2014), it is trivial that III(i) is a better minimizer of $\nu$ -- and therefore $\sigma$ -- than I(i). Likewise IV(ii) is a better minimizer of $\nu$ than I(ii), II(ii), and III(ii). Since tanh$(\nu a) <$ coth$(\nu a)$, we have that IV(ii) gives the smallest value of $\sigma$ compared with I(ii), II(ii), and III(ii).

\vspace{1mm}
\noindent It remains to compare II(i), III(i), IV(ii).

\begin{claim}
Fix $a \in (0,1]$. For a rectangle on $[-1,1] \times [-a,a]$, let the candidate for the 1st Steklov eigenvalue from class II(i) be $\sigma_1^{II(i)}$, and that from class III(i) be $\sigma_1^{III(i)}$. Then $\sigma_1^{III(i)} < \sigma_1^{II(i)}$.
\end{claim}

\begin{proof}
We want to show that III(i) is a better minimizer of $\sigma_1$ than II(i). For II(i), an eigenvalue $\sigma$ is given by an intersection of the graphs $\sigma = \nu$coth$(\nu)$ and $\sigma = \nu$cot$(\nu a)$. For III(i), an eigenvalue $\sigma$ is given by an intersection of the graphs $\sigma = \nu$tanh$(\nu)$ and $\sigma = \nu$cot$(\nu a)$.

\vspace{1mm}
\noindent Naturally, in either case, the candidate for $\sigma_1$ is the intersection which has the lowest height.

\vspace{1mm}
\noindent We show that the candidate for $\sigma_1$ is smaller in the case of III(i).

\FloatBarrier
\begin{figure}[!htbp]
\centering
\includegraphics[width=1\textwidth]{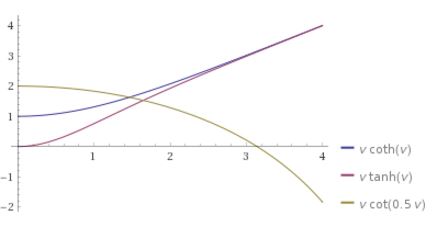}
\caption{\label{fig:IIIi}The graphs for $a = 0.5$.}
\end{figure}
\FloatBarrier

\vspace{1mm}
\noindent In each case, the lowest intersection must also be the intersection nearest to to the $\sigma$-axis, by Lemma \ref{increasing}. Hence we need only consider the graphs on $0<\nu< \nu_0$ where $\nu_0=\frac{\pi}{2a}$. Figure \ref{fig:IIIi} compares the lowest intersections for case III(i) and II(i). By the intermediate value theorem, there exists a candidate for $\sigma_1$ in $0<\nu< \nu_0$ for both cases III(i) and II(i). By the strict monotonicity of the graphs on this interval, these candidates are unique. Finally, the candidate from class III(i) has a lower height than the candidate from class II(i), because the graph of $\sigma = \nu$cot$(\nu a)$ is strictly decreasing on $0<\nu< \nu_0$. Otherwise we would derive a contradiction by the mean value theorem.

\end{proof}

\begin{claim}\label{iffsquare}
Fix $a \in (0,1]$. For a rectangle on $[-1,1] \times [-a,a]$, let the candidate for the 1st Steklov eigenvalue from class III(i) be $\sigma_1^{III(i)}$, and that from class IV(ii) be $\sigma_1^{IV(ii)}$. Then $\sigma_1^{IV(ii)} \leq \sigma_1^{III(i)}$, with equality iff $a = 1$.
\end{claim}

\begin{proof}
Similar to before, we want to show that IV(ii) is a better minimizer of $\sigma_1$ than III(i). For III(i), an eigenvalue $\sigma$ is given by an intersection of the graphs $\sigma = \nu$tanh$(\nu)$ and $\sigma =\nu$cot$(\nu a)$. For IV(ii), an eigenvalue $\sigma$ is given by an intersection of the graphs $\sigma =\nu$coth$(\nu a)$ and $\sigma =\nu$cot$(\nu)$.

\vspace{1mm}
\noindent We show that the candidate for $\sigma_1$ is smaller in the case of IV(ii).

\FloatBarrier
\begin{figure}[!htbp]
\centering
\includegraphics[width=1\textwidth]{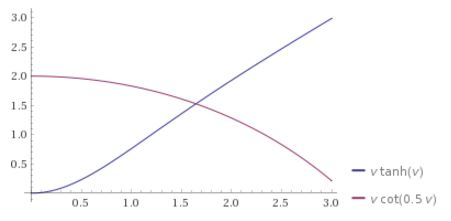}
\caption{\label{fig:III(i)}The candidate for $\sigma_1$ of Class III(i), in the case $a = 0.5$.}
\end{figure}
\FloatBarrier

\FloatBarrier
\begin{figure}[!htbp]
\centering
\includegraphics[width=1\textwidth]{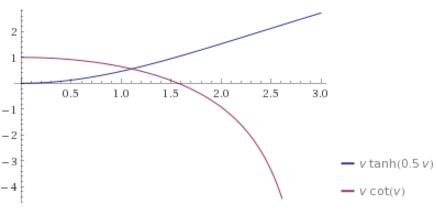}
\caption{\label{fig:IV(ii)}The candidate for $\sigma_1$ of Class IV(ii), in the case $a = 0.5$.}
\end{figure}
\FloatBarrier

\vspace{1mm}
\noindent In each case, the lowest intersection must also be the intersection nearest to to the $\sigma$-axis, by $(*)$. Hence we need only consider the graphs on $0<\nu< \nu_0$ where $\nu_0=\frac{\pi}{2a}$. Figures \ref{fig:III(i)} and \ref{fig:IV(ii)} give the lowest intersection for class III(i) and class IV(ii) respectively. These intersections can be shown to exist in $0<\nu< \nu_0$ by the intermediate value theorem. (A sample calculation for this is done explicitly in the proof of Claim \ref{diff}.) Since it is clear that both graphs in Figure \ref{fig:IV(ii)} are lower than the respective graphs in Figure \ref{fig:III(i)}, we can conclude that $\sigma_1^{IV(ii)} < \sigma_1^{III(i)}$.

\end{proof}

\vspace{1mm}
\noindent We have just shown that the 1st Steklov invariant on a given rectangle is \textit{always} given by eigenfunction IV(ii)! Moreover, for $a \neq 1$, this is the only eigenfunction in the 1st eigenspace, because $\sigma_1^{IV(ii)}$ is strictly smaller than any other eigenvalue. Let us state this as a theorem.

\begin{theorem}
On any rectangular domain that is not a square, the 1st Steklov eigenspace has basis $\{$sin$(\nu x)$cosh$(\nu y)$$\}$.
\end{theorem}

\vspace{3mm}
\noindent Below is a graph of this invariant against the length $a \in (0,1]$.

\FloatBarrier
\begin{figure}[!htbp]
\centering
\includegraphics[width=1\textwidth]{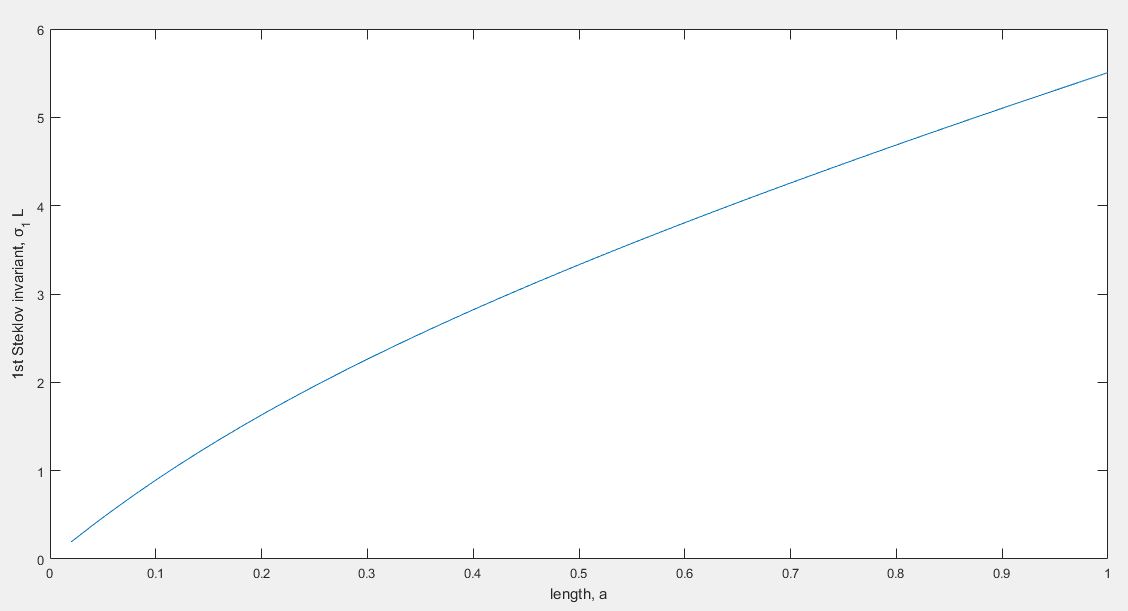}
\caption{\label{fig:smallest}The 1st Steklov invariant for a rectangle on $[-1,1]\times[-a,a]$.}
\end{figure}
\FloatBarrier

\vspace{3mm}
\noindent For completeness, we also plot the invariant given by each of the 8 eigenfunctions against $a$ in Figure \ref{fig:all}.

\FloatBarrier
\begin{figure}[!htbp]
\centering
\includegraphics[width=1\textwidth]{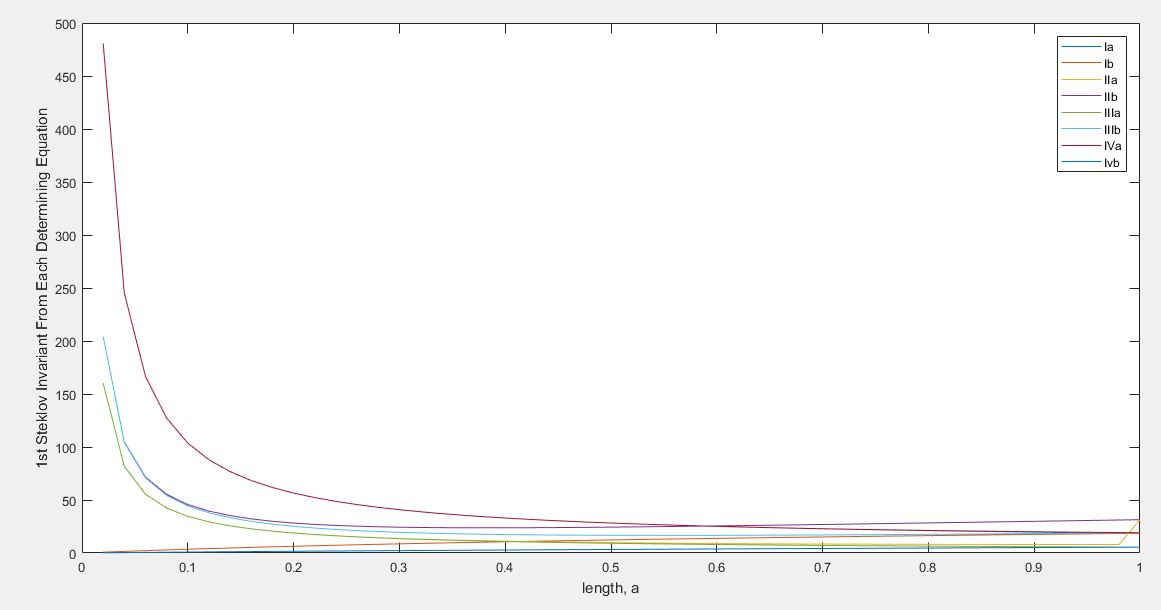}
\caption{\label{fig:all}The 1st Steklov invariant for each of 8 determining equations.}
\end{figure}
\FloatBarrier

\vspace{1mm}
\noindent The lowest graph in Figure \ref{fig:all} is, of course, the same graph as in Figure \ref{fig:smallest}.
Figure \ref{fig:smallest} suggests that $a = 1$ maximizes the 1st Steklov invariant. In fact, we now show that the graph is differentiable, increasing with $a$, and tends to the origin (in the first quadrant).

\begin{claim}\label{diff}
The graph in Figure \ref{fig:smallest} is differentiable on $a \in (0,1)$.
\end{claim}

\begin{proof}
Fix $a \in (0,1)$. Define $k_a:(0,\frac{\pi}{2}) \to \mathbb{R}$, $k_a(\nu) = \nu$cot$(\nu)-\nu$tanh$(\nu a)$. Since $\lim_{\nu \downarrow 0} k_a(\nu)=1>0$ and $\lim_{\nu \uparrow \frac{\pi}{2}} k_a(\nu)=-\frac{\pi}{2}\text{tanh}(\frac{\pi}{2} a) <0$, the intermediate value theorem guarantees some root $\nu_1^{IV(ii)} \in (0,\frac{\pi}{2})$. That is, cot$(\nu_1^{IV(ii)})-$tanh$(\nu_1^{IV(ii)} a)=0$. In fact, $\nu_1^{IV(ii)}$ is unique since $k_a$ is strictly decreasing. (Clearly $\sigma_1^{IV(ii)}=\nu_1^{IV(ii)}$tanh$(\nu_1^{IV(ii)}a)$.)
The above allows me to define $h:\mathbb{R}\times (0,\frac{\pi}{2}) \to \mathbb{R}$, $h(a,\nu)=k_a(\nu)$.
Next, we show that $\nu$ is a differentiable function of $a \in (0,1)$ by applying the implicit function theorem to $h$.
At each root $(a,\nu_1^{IV(ii)})$ of $h$, we check that $h_{\nu} \neq 0$:

\begin{align*} 
h_{\nu}( \ (a,\nu_1^{IV(ii)}) \ ) &= -\text{tanh}(\nu_1^{IV(ii)} a) - \nu_1^{IV(ii)} a \text{sech}^2(\nu_1^{IV(ii)} a) + \text{cot}(\nu_1^{IV(ii)}) - \nu_1^{IV(ii)} \text{cosec}^2(\nu_1^{IV(ii)}) \\ 
 &=  - \nu_1^{IV(ii)} a \text{sech}^2(\nu_1^{IV(ii)} a) - \nu_1^{IV(ii)} \text{cosec}^2(\nu_1^{IV(ii)}) \\
 &< 0.
\end{align*}

\vspace{1mm}
\noindent Therefore the implicit function theorem ensures that $\nu$ is locally a differentiable function of $a$ for each $a \in (0,1)$, and hence globally differentiable on $(0,1)$. It is also continuous on $[0,1]$. The composition $\sigma = \nu$tanh$(\nu a)$ is therefore a differentiable function of $a$.
\end{proof}

\begin{claim}
The graph in Figure \ref{fig:smallest} is strictly increasing on $(0,1] \ni a$. Moreover, $\sigma \downarrow 0$ as $a \downarrow 0$.
\end{claim}

\begin{proof}
Consider again Figure \ref{fig:IV(ii)}, which displays (for the special case $a=0.5$) the intersection of $\sigma=\nu $tanh$(\nu a)$ with $\sigma=\nu $cot$(\nu)$. Clearly as $a \downarrow 0$, the intersection height decreases (continuously!) to $0$. Moreover, for $a_1 < a_2$, we have $\sigma_1 < \sigma_2$.
\end{proof}

\vspace{1mm}
\noindent (Alternatively, by examining the \textit{Rayleigh quotient}, one can see easily that $\sigma \downarrow 0$ as $a \downarrow 0$.)

\section{Finding Steklov eigenfunctions on a cuboid}

Now we move to the 3D case.
The notation here is such that $(a,b,c)$ give the dimensions of a cube, but similarly to the 2D case, we may let $c=1$ be the longest side without loss of generality.

\vspace{1mm}
\noindent Let us consider a cuboid on $[-a,a]\times[-b,b]\times[-c,c]$.
we use separation of variables to solve Laplace's equation with a Steklov condition on the faces of the cuboid. This treatment will assume the completeness of the set of eigenfunctions obtained by separation of variables. For the 2D case, completeness can be shown using 'sloshing' methods (see Girouard and Polterovich, 2014).
\\ By symmetry, we only need to consider eigenfunctions which are odd or even in $x$, $y$, and $z$. This gives 8 parity classes of eigenfunctions. In the tables below, we need only consider $\lambda_i>0$.
\\ \\
1. CLASS 000: Even in $x$, $y$, $z$.

\begin{align*}
& Eigenfunction & & Eigenvalue \\\Xhline{5\arrayrulewidth}
          s(x,y,z) & =  1            & \sigma&=0 \\\hline
          s(x,y,z) & =  \text{cos}(\sqrt{\lambda_1^2+\lambda_2^2}x)\text{cosh}(\lambda_1y)\text{cosh}(\lambda_2z)            &  \sigma&= \lambda_1 \text{tanh} (\lambda_1b) = \lambda_2 \text{tanh} (\lambda_2c) = -\sqrt{\lambda_1^2 + \lambda_2^2}\text{tan}(\sqrt{\lambda_1^2 + \lambda_2^2}a) \\\hline
          s(x,y,z) & =  \text{cosh}(\lambda_1x)\text{cos}(\sqrt{\lambda_1^2+\lambda_2^2}y)\text{cosh}(\lambda_2z)            &  \sigma&= \lambda_1 \text{tanh} (\lambda_1a) = \lambda_2 \text{tanh} (\lambda_2c) = -\sqrt{\lambda_1^2 + \lambda_2^2}\text{tan}(\sqrt{\lambda_1^2 + \lambda_2^2}b) \\\hline
          s(x,y,z) & =  \text{cosh}(\lambda_1x)\text{cosh}(\lambda_2y)\text{cos}(\sqrt{\lambda_1^2+\lambda_2^2}z)            &  \sigma&= \lambda_1 \text{tanh} (\lambda_1a) = \lambda_2 \text{tanh} (\lambda_2b) = -\sqrt{\lambda_1^2 + \lambda_2^2}\text{tan}(\sqrt{\lambda_1^2 + \lambda_2^2}c) \\\hline
          s(x,y,z) & =  \text{cosh}(\sqrt{\lambda_1^2+\lambda_2^2}x)\text{cos}(\lambda_1y)\text{cos}(\lambda_2z)            &  \sigma&= -\lambda_1 \text{tan} (\lambda_1b) = -\lambda_2 \text{tan} (\lambda_2c) = \sqrt{\lambda_1^2 + \lambda_2^2}\text{tanh}(\sqrt{\lambda_1^2 + \lambda_2^2}a) \\\hline
          s(x,y,z) & =  \text{cos}(\lambda_1x)\text{cosh}(\sqrt{\lambda_1^2+\lambda_2^2}y)\text{cos}(\lambda_2z)            &  \sigma&= -\lambda_1 \text{tan} (\lambda_1a) = -\lambda_2 \text{tan} (\lambda_2c) = \sqrt{\lambda_1^2 + \lambda_2^2}\text{tanh}(\sqrt{\lambda_1^2 + \lambda_2^2}b) \\\hline
          s(x,y,z) & =  \text{cos}(\lambda_1x)\text{cos}(\lambda_2y)\text{cosh}(\sqrt{\lambda_1^2+\lambda_2^2}z)            &  \sigma&= -\lambda_1 \text{tan} (\lambda_1b) = -\lambda_2 \text{tan} (\lambda_2a) = \sqrt{\lambda_1^2 + \lambda_2^2}\text{tanh}(\sqrt{\lambda_1^2 + \lambda_2^2}c) \\\hline
\end{align*}

2. CLASS 001: Even in $x$, $y$, Odd in $z$.

\begin{align*}
& Eigenfunction & & Eigenvalue \\\Xhline{5\arrayrulewidth}
          s(x,y,z) & =  \text{cos}(\sqrt{\lambda_1^2+\lambda_2^2}x)\text{cosh}(\lambda_1y)\text{sinh}(\lambda_2z)            &  \sigma&= \lambda_1 \text{tanh} (\lambda_1b) = \lambda_2 \text{coth} (\lambda_2c) = -\sqrt{\lambda_1^2 + \lambda_2^2}\text{tan}(\sqrt{\lambda_1^2 + \lambda_2^2}a) \\\hline
          s(x,y,z) & =  \text{cosh}(\lambda_1x)\text{cos}(\sqrt{\lambda_1^2+\lambda_2^2}y)\text{sinh}(\lambda_2z)            &  \sigma&= \lambda_1 \text{tanh} (\lambda_1a) = \lambda_2 \text{coth} (\lambda_2c) = -\sqrt{\lambda_1^2 + \lambda_2^2}\text{tan}(\sqrt{\lambda_1^2 + \lambda_2^2}b) \\\hline
          s(x,y,z) & =  \text{cosh}(\lambda_1x)\text{cosh}(\lambda_2y)\text{sin}(\sqrt{\lambda_1^2+\lambda_2^2}z)            &  \sigma&= \lambda_1 \text{tanh} (\lambda_1a) = \lambda_2 \text{tanh} (\lambda_2b) = \sqrt{\lambda_1^2 + \lambda_2^2}\text{cot}(\sqrt{\lambda_1^2 + \lambda_2^2}c) \\\hline
          s(x,y,z) & =  \text{cosh}(\sqrt{\lambda_1^2+\lambda_2^2}x)\text{cos}(\lambda_1y)\text{sin}(\lambda_2z)            &  \sigma&= -\lambda_1 \text{tan} (\lambda_1b) =\lambda_2 \text{cot} (\lambda_2c) = \sqrt{\lambda_1^2 + \lambda_2^2}\text{tanh}(\sqrt{\lambda_1^2 + \lambda_2^2}a) \\\hline
          s(x,y,z) & =  \text{cos}(\lambda_1x)\text{cosh}(\sqrt{\lambda_1^2+\lambda_2^2}y)\text{sin}(\lambda_2z)            &  \sigma&= -\lambda_1 \text{tan} (\lambda_1a) = \lambda_2 \text{cot} (\lambda_2c) = \sqrt{\lambda_1^2 + \lambda_2^2}\text{tanh}(\sqrt{\lambda_1^2 + \lambda_2^2}b) \\\hline
          s(x,y,z) & =  \text{cos}(\lambda_1x)\text{cos}(\lambda_2y)\text{sinh}(\sqrt{\lambda_1^2+\lambda_2^2}z)            &  \sigma&= -\lambda_1 \text{tan} (\lambda_1a) = -\lambda_2 \text{tan} (\lambda_2b) = \sqrt{\lambda_1^2 + \lambda_2^2}\text{coth}(\sqrt{\lambda_1^2 + \lambda_2^2}c)
          \\\hline
          s(x,y,z) & =  z \text{cosh} (\lambda x) \text{cos} (\lambda y)           &  \sigma&= 1/c, \text{coth}(\lambda a) = \lambda c = -\text{cot}(\lambda b) \\\hline
          s(x,y,z) & =  z \text{cos} (\lambda x) \text{cosh} (\lambda y)           &  \sigma&= 1/c, \text{coth}(\lambda b) = \lambda c = -\text{cot}(\lambda a) \\\hline
\end{align*} 

3. CLASS 010: Even in $x$, $z$, Odd in $y$.

\begin{align*}
& Eigenfunction & & Eigenvalue \\\Xhline{5\arrayrulewidth}
          s(x,y,z) & =  \text{cos}(\sqrt{\lambda_1^2+\lambda_2^2}x)\text{sinh}(\lambda_2y)\text{cosh}(\lambda_1z)            &  \sigma&= \lambda_1 \text{tanh} (\lambda_1c) = \lambda_2 \text{coth} (\lambda_2b) = -\sqrt{\lambda_1^2 + \lambda_2^2}\text{tan}(\sqrt{\lambda_1^2 + \lambda_2^2}a) \\\hline
          s(x,y,z) & =  \text{cosh}(\lambda_1x)\text{sinh}(\lambda_2y) \text{cos}(\sqrt{\lambda_1^2+\lambda_2^2}z)           &  \sigma&= \lambda_1 \text{tanh} (\lambda_1 a) = \lambda_2 \text{coth} (\lambda_2b) = -\sqrt{\lambda_1^2 + \lambda_2^2}\text{tan}(\sqrt{\lambda_1^2 + \lambda_2^2}c) \\\hline
          s(x,y,z) & =  \text{cosh}(\lambda_1x)\text{sin}(\sqrt{\lambda_1^2+\lambda_2^2}y) \text{cosh}(\lambda_2z)           &  \sigma&= \lambda_1 \text{tanh} (\lambda_1 a) = \lambda_2 \text{tanh} (\lambda_2c) = \sqrt{\lambda_1^2 + \lambda_2^2}\text{cot}(\sqrt{\lambda_1^2 + \lambda_2^2}b) \\\hline
          s(x,y,z) & =  \text{cosh}(\sqrt{\lambda_1^2+\lambda_2^2}x)\text{sin}(\lambda_2y)\text{cos}(\lambda_1z)            &  \sigma&= -\lambda_1 \text{tan} (\lambda_1c) =\lambda_2 \text{cot} (\lambda_2b) = \sqrt{\lambda_1^2 + \lambda_2^2}\text{tanh}(\sqrt{\lambda_1^2 + \lambda_2^2}a) \\\hline
          s(x,y,z) & =  \text{cos}(\lambda_1x)\text{sin}(\lambda_2y)  \text{cosh}(\sqrt{\lambda_1^2+\lambda_2^2}z)          &  \sigma&= -\lambda_1 \text{tan} (\lambda_1 a) = \lambda_2 \text{cot} (\lambda_2b) = \sqrt{\lambda_1^2 + \lambda_2^2}\text{tanh}(\sqrt{\lambda_1^2 + \lambda_2^2}c) \\\hline
          s(x,y,z) & =  \text{cos}(\lambda_1x) \text{sinh}(\sqrt{\lambda_1^2+\lambda_2^2}y)      \text{cos}(\lambda_2z)       &  \sigma&= -\lambda_1 \text{tan} (\lambda_1a) = -\lambda_2 \text{tan} (\lambda_2c) = \sqrt{\lambda_1^2 + \lambda_2^2}\text{coth}(\sqrt{\lambda_1^2 + \lambda_2^2}b)
          \\\hline
          s(x,y,z) & =  y \text{cosh} (\lambda x) \text{cos} (\lambda z)           &  \sigma&= 1/b, \text{coth}(\lambda a) = \lambda b = -\text{cot}(\lambda c) \\\hline
          s(x,y,z) & =  y \text{cos} (\lambda x) \text{cosh} (\lambda z)           &  \sigma&= 1/b, \text{coth}(\lambda c) = \lambda b = -\text{cot}(\lambda a) \\\hline
\end{align*}

4. CLASS 011: Odd in $y$, $z$, Even in $x$.

\begin{align*}
& Eigenfunction & & Eigenvalue \\\Xhline{5\arrayrulewidth}
          s(x,y,z) & =  \text{cosh}(\lambda_2x)\text{sin}(\sqrt{\lambda_1^2+\lambda_2^2}y) \text{sinh}(\lambda_1z)           &  \sigma&= \lambda_1 \text{coth} (\lambda_1c) = \lambda_2 \text{tanh} (\lambda_2a) = \sqrt{\lambda_1^2 + \lambda_2^2}\text{cot}(\sqrt{\lambda_1^2 + \lambda_2^2}b) \\\hline
          s(x,y,z) & =  \text{cosh}(\lambda_2x)\text{sinh}(\lambda_1y) \text{sin}(\sqrt{\lambda_1^2+\lambda_2^2}z)           &  \sigma&= \lambda_1 \text{coth} (\lambda_1 b) = \lambda_2 \text{tanh} (\lambda_2 a) = \sqrt{\lambda_1^2 + \lambda_2^2}\text{cot}(\sqrt{\lambda_1^2 + \lambda_2^2}c) \\\hline
          s(x,y,z) & =  \text{cos}(\sqrt{\lambda_1^2+\lambda_2^2}x)\text{sinh}(\lambda_1y) \text{sinh}(\lambda_2z)           &  \sigma&= \lambda_1 \text{coth} (\lambda_1 b) = \lambda_2 \text{coth} (\lambda_2c) = - \sqrt{\lambda_1^2 + \lambda_2^2}\text{tan}(\sqrt{\lambda_1^2 + \lambda_2^2}a) \\\hline
          s(x,y,z) & =  \text{cos}(\lambda_2x)\text{sinh}(\sqrt{\lambda_1^2+\lambda_2^2}y) \text{sin}(\lambda_1z)           &  \sigma&= \lambda_1 \text{cot} (\lambda_1c) = - \lambda_2 \text{tan} (\lambda_2 a) = \sqrt{\lambda_1^2 + \lambda_2^2}\text{coth}(\sqrt{\lambda_1^2 + \lambda_2^2} b) \\\hline
          s(x,y,z) & =  \text{cos}(\lambda_2x)\text{sin}(\lambda_1y)\text{sinh}(\sqrt{\lambda_1^2+\lambda_2^2}z)            &  \sigma&= \lambda_1 \text{cot} (\lambda_1 b) = - \lambda_2 \text{tan} (\lambda_2 a) = \sqrt{\lambda_1^2 + \lambda_2^2}\text{coth}(\sqrt{\lambda_1^2 + \lambda_2^2}c) \\\hline
          s(x,y,z) & =  \text{cosh}(\sqrt{\lambda_1^2+\lambda_2^2}x) \text{sin}(\lambda_1y)   \text{sin}(\lambda_2z)         &  \sigma&= \lambda_1 \text{cot} (\lambda_1 b) = \lambda_2 \text{cot} (\lambda_2c) = \sqrt{\lambda_1^2 + \lambda_2^2}\text{tanh}(\sqrt{\lambda_1^2 + \lambda_2^2}a)
          \\\hline
          s(x,y,z) & =  y \text{cos} (\lambda x) \text{sinh} (\lambda z)          &  \sigma&= 1/b, \text{tanh}(\lambda c) = \lambda b = -\text{cot}(\lambda a) \\\hline
          s(x,y,z) & =  y \text{cosh} (\lambda x)   \text{sin} (\lambda z)         &  \sigma&= 1/b, \text{tan}(\lambda c) = \lambda b = \text{coth}(\lambda a) \\\hline
          s(x,y,z) & =  z \text{cos} (\lambda x) \text{sinh} (\lambda y)            &  \sigma&= 1/c, \text{tanh}(\lambda b) = \lambda c = -\text{cot}(\lambda a) \\\hline
          s(x,y,z) & =  z \text{cosh} (\lambda x) \text{sin} (\lambda y)           &  \sigma&= 1/c, \text{tan}(\lambda b) = \lambda c = \text{coth}(\lambda a) \\\hline
\end{align*}

5. CLASS 100: Even in $y$, $z$, Odd in $x$.

\begin{align*}
& Eigenfunction & & Eigenvalue \\\Xhline{5\arrayrulewidth}
          s(x,y,z) & =  \text{sinh}(\lambda_2x) \text{cosh}(\lambda_1y) \text{cos}(\sqrt{\lambda_1^2+\lambda_2^2}z)         &  \sigma&= \lambda_1 \text{tanh} (\lambda_1b) = \lambda_2 \text{coth} (\lambda_2 a) = -\sqrt{\lambda_1^2 + \lambda_2^2}\text{tan}(\sqrt{\lambda_1^2 + \lambda_2^2} c) \\\hline
          s(x,y,z) & =  \text{sinh}(\lambda_2x) \text{cos}(\sqrt{\lambda_1^2+\lambda_2^2}y)  \text{cosh}(\lambda_1z)          &  \sigma&= \lambda_1 \text{tanh} (\lambda_1 c) = \lambda_2 \text{coth} (\lambda_2a) = -\sqrt{\lambda_1^2 + \lambda_2^2}\text{tan}(\sqrt{\lambda_1^2 + \lambda_2^2}b) \\\hline
          s(x,y,z) & =  \text{sin}(\sqrt{\lambda_1^2+\lambda_2^2}x)\text{cosh}(\lambda_2y)\text{cosh}(\lambda_1z)            &  \sigma&= \lambda_1 \text{tanh} (\lambda_1 c) = \lambda_2 \text{tanh} (\lambda_2b) = \sqrt{\lambda_1^2 + \lambda_2^2}\text{cot}(\sqrt{\lambda_1^2 + \lambda_2^2}a) \\\hline
          s(x,y,z) & =  \text{sin}(\lambda_2x)\text{cos}(\lambda_1y)\text{cosh}(\sqrt{\lambda_1^2+\lambda_2^2}z)            &  \sigma&= -\lambda_1 \text{tan} (\lambda_1b) =\lambda_2 \text{cot} (\lambda_2 a) = \sqrt{\lambda_1^2 + \lambda_2^2}\text{tanh}(\sqrt{\lambda_1^2 + \lambda_2^2}c) \\\hline
          s(x,y,z) & =  \text{sin}(\lambda_2x)\text{cosh}(\sqrt{\lambda_1^2+\lambda_2^2}y)\text{cos}(\lambda_1z)            &  \sigma&= -\lambda_1 \text{tan} (\lambda_1 c) = \lambda_2 \text{cot} (\lambda_2 a) = \sqrt{\lambda_1^2 + \lambda_2^2}\text{tanh}(\sqrt{\lambda_1^2 + \lambda_2^2}b) \\\hline
          s(x,y,z) & =  \text{sinh}(\sqrt{\lambda_1^2+\lambda_2^2}x)\text{cos}(\lambda_2y)\text{cos}(\lambda_1z)            &  \sigma&= -\lambda_1 \text{tan} (\lambda_1 c) = -\lambda_2 \text{tan} (\lambda_2b) = \sqrt{\lambda_1^2 + \lambda_2^2}\text{coth}(\sqrt{\lambda_1^2 + \lambda_2^2}a)
          \\\hline
          s(x,y,z) & =  x \text{cos} (\lambda y) \text{cosh} (\lambda z)           &  \sigma&= 1/a, \text{coth}(\lambda c) = \lambda a = -\text{cot}(\lambda b) \\\hline
          s(x,y,z) & =  x \text{cosh} (\lambda y) \text{cos} (\lambda z)          &  \sigma&= 1/a, \text{coth}(\lambda b) = \lambda a = -\text{cot}(\lambda c) \\\hline
\end{align*}

6. CLASS 101: Odd in $x$, $z$, Even in $y$.

\begin{align*}
& Eigenfunction & & Eigenvalue \\\Xhline{5\arrayrulewidth}
          s(x,y,z) & =  \text{sin}(\sqrt{\lambda_1^2+\lambda_2^2}x)\text{cosh}(\lambda_2y) \text{sinh}(\lambda_1z)           &  \sigma&= \lambda_1 \text{coth} (\lambda_1c) = \lambda_2 \text{tanh} (\lambda_2b) = \sqrt{\lambda_1^2 + \lambda_2^2}\text{cot}(\sqrt{\lambda_1^2 + \lambda_2^2} a) \\\hline
          s(x,y,z) & =  \text{sinh}(\lambda_1x)\text{cosh}(\lambda_2y) \text{sin}(\sqrt{\lambda_1^2+\lambda_2^2}z)           &  \sigma&= \lambda_1 \text{coth} (\lambda_1 a) = \lambda_2 \text{tanh} (\lambda_2b) = \sqrt{\lambda_1^2 + \lambda_2^2}\text{cot}(\sqrt{\lambda_1^2 + \lambda_2^2}c) \\\hline
          s(x,y,z) & =  \text{sinh}(\lambda_1x)\text{cos}(\sqrt{\lambda_1^2+\lambda_2^2}y) \text{sinh}(\lambda_2z)           &  \sigma&= \lambda_1 \text{coth} (\lambda_1 a) = \lambda_2 \text{coth} (\lambda_2c) = - \sqrt{\lambda_1^2 + \lambda_2^2}\text{tan}(\sqrt{\lambda_1^2 + \lambda_2^2}b) \\\hline
          s(x,y,z) & =  \text{sinh}(\sqrt{\lambda_1^2+\lambda_2^2}x)\text{cos}(\lambda_2y) \text{sin}(\lambda_1z)           &  \sigma&= \lambda_1 \text{cot} (\lambda_1c) = - \lambda_2 \text{tan} (\lambda_2b) = \sqrt{\lambda_1^2 + \lambda_2^2}\text{coth}(\sqrt{\lambda_1^2 + \lambda_2^2}a) \\\hline
          s(x,y,z) & =  \text{sin}(\lambda_1x)\text{cos}(\lambda_2y)\text{sinh}(\sqrt{\lambda_1^2+\lambda_2^2}z)            &  \sigma&= \lambda_1 \text{cot} (\lambda_1 a) = - \lambda_2 \text{tan} (\lambda_2b) = \sqrt{\lambda_1^2 + \lambda_2^2}\text{coth}(\sqrt{\lambda_1^2 + \lambda_2^2}c) \\\hline
          s(x,y,z) & =  \text{sin}(\lambda_1x)\text{cosh}(\sqrt{\lambda_1^2+\lambda_2^2}y)   \text{sin}(\lambda_2z)         &  \sigma&= \lambda_1 \text{cot} (\lambda_1a) = \lambda_2 \text{cot} (\lambda_2c) = \sqrt{\lambda_1^2 + \lambda_2^2}\text{tanh}(\sqrt{\lambda_1^2 + \lambda_2^2}b)
          \\\hline
          s(x,y,z) & =  x \text{cos} (\lambda y) \text{sinh} (\lambda z)          &  \sigma&= 1/a, \text{tanh}(\lambda c) = \lambda a = -\text{cot}(\lambda b) \\\hline
          s(x,y,z) & =  x \text{cosh} (\lambda y)   \text{sin} (\lambda z)         &  \sigma&= 1/a, \text{tan}(\lambda c) = \lambda a = \text{coth}(\lambda b) \\\hline
          s(x,y,z) & =  z \text{sinh} (\lambda x) \text{cos} (\lambda y)           &  \sigma&= 1/c, \text{tanh}(\lambda a) = \lambda c = -\text{cot}(\lambda b) \\\hline
          s(x,y,z) & =  z \text{sin} (\lambda x) \text{cosh} (\lambda y)           &  \sigma&= 1/c, \text{tan}(\lambda a) = \lambda c = \text{coth}(\lambda b) \\\hline
\end{align*}

7. CLASS 110: Odd in $x$, $y$, Even in $z$.

\begin{align*}
& Eigenfunction & & Eigenvalue \\\Xhline{5\arrayrulewidth}
          s(x,y,z) & =  \text{sin}(\sqrt{\lambda_1^2+\lambda_2^2}x)\text{sinh}(\lambda_1y)\text{cosh}(\lambda_2z)            &  \sigma&= \lambda_1 \text{coth} (\lambda_1b) = \lambda_2 \text{tanh} (\lambda_2c) = \sqrt{\lambda_1^2 + \lambda_2^2}\text{cot}(\sqrt{\lambda_1^2 + \lambda_2^2}a) \\\hline
          s(x,y,z) & =  \text{sinh}(\lambda_1x)\text{sin}(\sqrt{\lambda_1^2+\lambda_2^2}y)\text{cosh}(\lambda_2z)            &  \sigma&= \lambda_1 \text{coth} (\lambda_1 a) = \lambda_2 \text{tanh} (\lambda_2c) = \sqrt{\lambda_1^2 + \lambda_2^2}\text{cot}(\sqrt{\lambda_1^2 + \lambda_2^2}b) \\\hline
          s(x,y,z) & =  \text{sinh}(\lambda_1x)\text{sinh}(\lambda_2y)\text{cos}(\sqrt{\lambda_1^2+\lambda_2^2}z)            &  \sigma&= \lambda_1 \text{coth} (\lambda_1 a) = \lambda_2 \text{coth} (\lambda_2b) = - \sqrt{\lambda_1^2 + \lambda_2^2}\text{tan}(\sqrt{\lambda_1^2 + \lambda_2^2}c) \\\hline
          s(x,y,z) & =  \text{sinh}(\sqrt{\lambda_1^2+\lambda_2^2}x)\text{sin}(\lambda_1y)\text{cos}(\lambda_2z)            &  \sigma&= \lambda_1 \text{cot} (\lambda_1b) = - \lambda_2 \text{tan} (\lambda_2c) = \sqrt{\lambda_1^2 + \lambda_2^2}\text{coth}(\sqrt{\lambda_1^2 + \lambda_2^2}a) \\\hline
          s(x,y,z) & =  \text{sin}(\lambda_1x)\text{sinh}(\sqrt{\lambda_1^2+\lambda_2^2}y)\text{cos}(\lambda_2z)            &  \sigma&= \lambda_1 \text{cot} (\lambda_1 a) = - \lambda_2 \text{tan} (\lambda_2c) = \sqrt{\lambda_1^2 + \lambda_2^2}\text{coth}(\sqrt{\lambda_1^2 + \lambda_2^2}b) \\\hline
          s(x,y,z) & =  \text{sin}(\lambda_1x)\text{sin}(\lambda_2y)\text{cosh}(\sqrt{\lambda_1^2+\lambda_2^2}z)            &  \sigma&= \lambda_1 \text{cot} (\lambda_1a) = \lambda_2 \text{cot} (\lambda_2b) = \sqrt{\lambda_1^2 + \lambda_2^2}\text{tanh}(\sqrt{\lambda_1^2 + \lambda_2^2}c)
          \\\hline
          s(x,y,z) & =  x \text{sinh} (\lambda y) \text{cos} (\lambda z)           &  \sigma&= 1/a, \text{tanh}(\lambda b) = \lambda a = -\text{cot}(\lambda c) \\\hline
          s(x,y,z) & =  x \text{sin} (\lambda y) \text{cosh} (\lambda z)           &  \sigma&= 1/a, \text{tan}(\lambda b) = \lambda a = \text{coth}(\lambda c) \\\hline
          s(x,y,z) & =  y \text{sinh} (\lambda x) \text{cos} (\lambda z)           &  \sigma&= 1/b, \text{tanh}(\lambda a) = \lambda b = -\text{cot}(\lambda c) \\\hline
          s(x,y,z) & =  y \text{sin} (\lambda x) \text{cosh} (\lambda z)           &  \sigma&= 1/b, \text{tan}(\lambda a) = \lambda b = \text{coth}(\lambda c) \\\hline
\end{align*}

8. CLASS 111: Odd in $x$, $y$, $z$.

\begin{align*}
& Eigenfunction & & Eigenvalue \\\Xhline{5\arrayrulewidth}
          s(x,y,z) & =  xyz             & \sigma  &=1/a =1/b  =1/c \\\hline
          s(x,y,z) & =  \text{sin}(\sqrt{\lambda_1^2+\lambda_2^2}x)\text{sinh}(\lambda_1y)\text{sinh}(\lambda_2z)            &  \sigma&= \lambda_1 \text{coth} (\lambda_1b) = \lambda_2 \text{coth} (\lambda_2c) = \sqrt{\lambda_1^2 + \lambda_2^2}\text{cot}(\sqrt{\lambda_1^2 + \lambda_2^2}a) \\\hline
          s(x,y,z) & =  \text{sinh}(\lambda_1x)\text{sin}(\sqrt{\lambda_1^2+\lambda_2^2}y)\text{sinh}(\lambda_2z)            &  \sigma&= \lambda_1 \text{coth} (\lambda_1a) = \lambda_2 \text{coth} (\lambda_2c) = \sqrt{\lambda_1^2 + \lambda_2^2}\text{cot}(\sqrt{\lambda_1^2 + \lambda_2^2}b) \\\hline
          s(x,y,z) & =  \text{sinh}(\lambda_1x)\text{sinh}(\lambda_2y)\text{sin}(\sqrt{\lambda_1^2+\lambda_2^2}z)            &  \sigma&= \lambda_1 \text{coth} (\lambda_1a) = \lambda_2 \text{coth} (\lambda_2b) = \sqrt{\lambda_1^2 + \lambda_2^2}\text{cot}(\sqrt{\lambda_1^2 + \lambda_2^2}c) \\\hline
          s(x,y,z) & =  \text{sinh}(\sqrt{\lambda_1^2+\lambda_2^2}x)\text{sin}(\lambda_1y)\text{sin}(\lambda_2z)            &  \sigma&= \lambda_1 \text{cot} (\lambda_1b) = \lambda_2 \text{cot} (\lambda_2c) = \sqrt{\lambda_1^2 + \lambda_2^2}\text{coth}(\sqrt{\lambda_1^2 + \lambda_2^2}a) \\\hline
          s(x,y,z) & =  \text{sin}(\lambda_1x)\text{sinh}(\sqrt{\lambda_1^2+\lambda_2^2}y)\text{sin}(\lambda_2z)            &  \sigma&= \lambda_1 \text{cot} (\lambda_1 a) = \lambda_2 \text{cot} (\lambda_2c) = \sqrt{\lambda_1^2 + \lambda_2^2}\text{coth}(\sqrt{\lambda_1^2 + \lambda_2^2}b) \\\hline
          s(x,y,z) & =  \text{sin}(\lambda_1x)\text{sin}(\lambda_2y)\text{sinh}(\sqrt{\lambda_1^2+\lambda_2^2}z)            &  \sigma&= \lambda_1 \text{cot} (\lambda_1b) = \lambda_2 \text{cot} (\lambda_2a) = \sqrt{\lambda_1^2 + \lambda_2^2}\text{coth}(\sqrt{\lambda_1^2 + \lambda_2^2}c)
\\\hline
		  s(x,y,z) & =  x \text{sinh} (\lambda y) \text{sin} (\lambda z)            &  \sigma&= 1/a,  \text{tanh}(\lambda b) = \lambda a = \text{tan}(\lambda c) \\\hline
          s(x,y,z) & =  x \text{sin} (\lambda y) \text{sinh} (\lambda z)            &  \sigma&= 1/a,  \text{tan}(\lambda b) = \lambda a = \text{tanh}(\lambda c) \\\hline
          s(x,y,z) & =  y \text{sinh} (\lambda x) \text{sin} (\lambda z)            &  \sigma&= 1/b,  \text{tanh}(\lambda a) = \lambda b = \text{tan}(\lambda c) \\\hline
          s(x,y,z) & =  y \text{sin} (\lambda x) \text{sinh} (\lambda z)            &  \sigma&= 1/b,  \text{tan}(\lambda a) = \lambda b = \text{tanh}(\lambda c) \\\hline
          s(x,y,z) & =  z \text{sinh} (\lambda x) \text{sin} (\lambda y)            &  \sigma&= 1/c,  \text{tanh}(\lambda a) = \lambda c = \text{tan}(\lambda b) \\\hline
          s(x,y,z) & =  z \text{sin} (\lambda x) \text{sinh} (\lambda y)            &  \sigma&= 1/c,  \text{tan}(\lambda a) = \lambda c = \text{tanh}(\lambda b) \\\hline
\end{align*} 

\vspace{1mm}
\noindent We may plot an analogue of Figure \ref{fig:smallest} by considering the cuboid with $c=1$ and $a,b \leq 1$ without loss of generality.

\FloatBarrier
\begin{figure}[!htbp]
\centering
\includegraphics[width=1\textwidth]{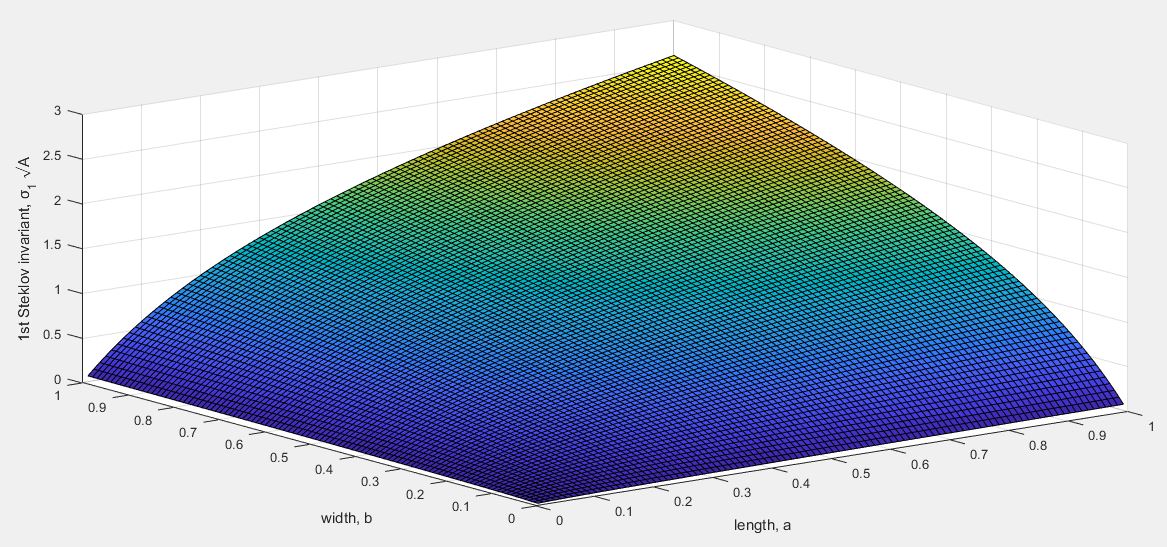}
\caption{\label{fig:smallestoncuboid}The 1st Steklov invariant for a cuboid on $[-a,a]\times[-b,b]\times[-1,1]$.}
\end{figure}
\FloatBarrier

\vspace{1mm}
\noindent Drawing similarities with the 2D scenario, it would seem likely that the following holds:
\begin{conjecture}
The surface in Figure \ref{fig:smallestoncuboid} is differentiable and tends to the origin (in the 1st octant).
\end{conjecture}

\section{Analysis of the 1st Steklov eigenspace on the cuboid}
We may assume, without loss of generality, that $0 \leq a \leq b \leq c=1$.
Drawing similarities with the 2D scenario, it is likely that the following holds:

\begin{conjecture}
$\sigma_1$ is always due to the same one eigenfunction, namely cosh$(\lambda_1 x)$cosh$(\lambda_2 y)$sin$(\sqrt{\lambda_1^2+\lambda_2^2}z)$. Further, if $a<b<1$, then this is the unique eigenfunction which gives $\sigma_1$. Otherwise, there may be other eigenfunctions that give this eigenvalue and $\sigma_1$ becomes a multiple eigenvalue.
\end{conjecture}

\section{The square and the cube}
Now we will perform some explicit numerical calculations for the square and the cube.

\vspace{1mm}
\noindent It is important to note that for the square on $[-1,1]\times[-a,a]$ and the cube on $[-a,a]\times[-b,b]\times[-1,1]$, there are redundancies in the determining equations (sometimes across parity classes); however the eigenfunctions are still distinct, which gives rise to an eigenvalue being multiple. For instance, for the square, parity classes I(i) and I(ii) give the same eigenvalue, as do classes III(i) and IV(ii). The latter is an example of multiplicity across classes.

\vspace{1mm}
\noindent Let us collect, for the square, the eigenfunctions according to determining equation, along with their corresponding candidate for $\sigma_1$. We include the eigenfunction $s(x,y)=xy$.

\FloatBarrier
\begin{table}[!htbp]
\centering
\begin{tabular}{l|c|c|r}
Eigenfunctions & Determining Equation & Eigenvalue $\sigma$\\\hline\hline
cosh$(\nu x)$cos$(\nu y)$, & tan$(\nu)+$tanh$(\nu)=0$ & $\nu$tanh$(\nu)$ \\
cos$(\nu x)$cosh$(\nu y)$ & Candidate for $\nu_1 = 2.3650203...$ & Candidate for $\sigma_1 = 2.3236377...$ \\\hline
sinh$(\nu x)$sin$(\nu y)$, & tan$(\nu)=$tanh$(\nu)$ & $\nu$coth$(\nu)$ \\
sin$(\nu x)$sinh$(\nu y)$ & Candidate for $\nu_1 = 3.9266023...$ & Candidate for  $\sigma_1 = 3.9296545...$ \\\hline
cosh$(\nu x)$sin$(\nu y)$, & tan$(\nu)=$coth$(\nu)$ & $\nu$tanh$(\nu)$ \\
sin$(\nu x)$cosh$(\nu y)$ & Candidate for $\nu_1 = 0.9375520...$ & Candidate for  $\sigma_1 = 0.6882527...$ \\\hline
sinh$(\nu x)$cos$(\nu y)$, & tan$(\nu)+$coth$(\nu)=0$ & $\nu$coth$(\nu)$ \\
cos$(\nu x)$sinh$(\nu y)$ & Candidate for $\nu_1 = 2.3470455...$ & Candidate for  $\sigma_1 = 2.3903892...$ \\\hline
$xy$ & - & $1$ \\
&  & Candidate for  $\sigma_1 = 1$
\end{tabular}
\caption{\label{tablesquare}For the square, eigenfunctions collect into pairs.}
\end{table}
\FloatBarrier

\vspace{1mm}
\noindent Therefore, comparing all the candidates in Table \ref{tablesquare}, we see that the correct value for $\sigma_1$ is $0.6882527...$, whose two eigenfunctions are cosh$(\nu x)$sin$(\nu y)$ and sin$(\nu x)$cosh$(\nu y)$. The perimeter of our square is 8 units, so the first Steklov invariant for any square is $\sigma_1L=8\sigma_1=5.506...$, which, as one can double-check, is the maximum in Figure \ref{fig:smallest}.

\vspace{1mm}
\noindent (Of course, we already know from the previous section that class IV(ii) must give $\sigma_1$. However, since $a=1$, class III(i) also gives $\sigma_1$. See Claim \ref{iffsquare}.)

\vspace{1mm}
\noindent We perform a similar calculation for the cube, collecting the Steklov eigenfunctions according to the candidate for $\sigma_1$ which they produce. For brevity, we omit those eigenfunctions such as $s(x,y,z)=xyz$ which give $\sigma_1 = 1$; the following table lists only the cases with non-degenerate determining equations.

\FloatBarrier
\begin{table}[!htbp]
\centering
\begin{tabular}{l|c|c|r}
Eigenfunctions & Determining Equation & Eigenvalue $\sigma$\\\hline\hline
cosh$(\lambda_1 x)$cosh$(\lambda_2 y)$cos$(\sqrt{\lambda_1^2+\lambda_2^2}z)$ & $\lambda_1$tanh$\lambda_1$ = $\lambda_2$tanh$\lambda_2$ = $-\sqrt{\lambda_1^2+\lambda_2^2}$tan$(\sqrt{\lambda_1^2+\lambda_2^2})$ & $\lambda_1$tanh$\lambda_1$ \\

cosh$(\lambda_1 x)$cosh$(\lambda_2 z)$cos$(\sqrt{\lambda_1^2+\lambda_2^2}y)$ & Candidate for $((\lambda_1)_1,(\lambda_2)_1)$ = $(1.8041319,1.8041319)$ & Candidate for $\sigma_1$ = 1.7089319 \\

cosh$(\lambda_1 y)$cosh$(\lambda_2 z)$cos$(\sqrt{\lambda_1^2+\lambda_2^2}x)$ & & \\\hline
cos$(\lambda_1 x)$cos$(\lambda_2 y)$cosh$(\sqrt{\lambda_1^2+\lambda_2^2}z)$ & $-\lambda_1$tan$\lambda_1$ = $-\lambda_2$tan$\lambda_2$ = $\sqrt{\lambda_1^2+\lambda_2^2}$tanh$(\sqrt{\lambda_1^2+\lambda_2^2})$ & $\sqrt{\lambda_1^2+\lambda_2^2}$tanh$(\sqrt{\lambda_1^2+\lambda_2^2})$ \\

cos$(\lambda_1 x)$cos$(\lambda_2 z)$cosh$(\sqrt{\lambda_1^2+\lambda_2^2}y)$ & Candidate for $((\lambda_1)_1,(\lambda_2)_1)$ = $(2.1882115, 2.1882115)$ & Candidate for $\sigma_1$ = 3.0819274 \\

cos$(\lambda_1 y)$cos$(\lambda_2 z)$cosh$(\sqrt{\lambda_1^2+\lambda_2^2}x)$ & & \\\hline
cosh$(\lambda_1 x)$sinh$(\lambda_2 y)$cos$(\sqrt{\lambda_1^2+\lambda_2^2}z)$ & $\lambda_1$tanh$\lambda_1$ = $\lambda_2$coth$\lambda_2$ = $-\sqrt{\lambda_1^2+\lambda_2^2}$tan$(\sqrt{\lambda_1^2+\lambda_2^2})$ & $\lambda_1$tanh$\lambda_1$ \\

cosh$(\lambda_1 x)$sinh$(\lambda_2 z)$cos$(\sqrt{\lambda_1^2+\lambda_2^2}y)$ & Candidate for $((\lambda_1)_1,(\lambda_2)_1)$ = $(1.883677,1.677149)$ & Candidate for $\sigma_1$ = 1.7985693 \\

cosh$(\lambda_1 y)$sinh$(\lambda_2 z)$cos$(\sqrt{\lambda_1^2+\lambda_2^2}x)$ & & \\
sinh$(\lambda_1 x)$cosh$(\lambda_2 y)$cos$(\sqrt{\lambda_1^2+\lambda_2^2}z)$ & & \\
sinh$(\lambda_1 x)$cosh$(\lambda_2 z)$cos$(\sqrt{\lambda_1^2+\lambda_2^2}y)$ & & \\
sinh$(\lambda_1 y)$cosh$(\lambda_2 z)$cos$(\sqrt{\lambda_1^2+\lambda_2^2}x)$ & & \\\hline
cosh$(\lambda_1 x)$cosh$(\lambda_2 y)$sin$(\sqrt{\lambda_1^2+\lambda_2^2}z)$ & $\lambda_1$tanh$\lambda_1$ = $\lambda_2$tanh$\lambda_2$ = $\sqrt{\lambda_1^2+\lambda_2^2}$cot$(\sqrt{\lambda_1^2+\lambda_2^2})$ & $\lambda_1$tanh$\lambda_1$ \\

cosh$(\lambda_1 x)$cosh$(\lambda_2 z)$sin$(\sqrt{\lambda_1^2+\lambda_2^2}y)$ & Candidate for $((\lambda_1)_1,(\lambda_2)_1)$ = $(,)$ & Candidate for $\sigma_1$ = 0.5315091 \\

cosh$(\lambda_1 y)$cosh$(\lambda_2 z)$sin$(\sqrt{\lambda_1^2+\lambda_2^2}x)$ & & \\\hline
cos$(\lambda_1 x)$sin$(\lambda_2 y)$cosh$(\sqrt{\lambda_1^2+\lambda_2^2}z)$ & $-\lambda_1$tan$\lambda_1$ = $\lambda_2$cot$\lambda_2$ = $\sqrt{\lambda_1^2+\lambda_2^2}$tanh$(\sqrt{\lambda_1^2+\lambda_2^2})$ & $\sqrt{\lambda_1^2+\lambda_2^2}$tanh$(\sqrt{\lambda_1^2+\lambda_2^2})$ \\

cos$(\lambda_1 x)$sin$(\lambda_2 z)$cosh$(\sqrt{\lambda_1^2+\lambda_2^2}y)$ & Candidate for $((\lambda_1)_1,(\lambda_2)_1)$ = $(2.0017440,3.8679675)$ & Candidate for $\sigma_1$ = 4.3538085 \\

cos$(\lambda_1 y)$sin$(\lambda_2 z)$cosh$(\sqrt{\lambda_1^2+\lambda_2^2}x)$ & & \\
sin$(\lambda_1 x)$cos$(\lambda_2 y)$cosh$(\sqrt{\lambda_1^2+\lambda_2^2}z)$ & & \\
sin$(\lambda_1 x)$cos$(\lambda_2 z)$cosh$(\sqrt{\lambda_1^2+\lambda_2^2}y)$ & & \\
sin$(\lambda_1 y)$cos$(\lambda_2 z)$cosh$(\sqrt{\lambda_1^2+\lambda_2^2}x)$ & & \\\hline
cos$(\lambda_1 x)$cos$(\lambda_2 y)$sinh$(\sqrt{\lambda_1^2+\lambda_2^2}z)$ & $-\lambda_1$tan$\lambda_1$ = $-\lambda_2$tan$\lambda_2$ = $\sqrt{\lambda_1^2+\lambda_2^2}$coth$(\sqrt{\lambda_1^2+\lambda_2^2})$ & $\sqrt{\lambda_1^2+\lambda_2^2}$coth$(\sqrt{\lambda_1^2+\lambda_2^2})$ \\

cos$(\lambda_1 x)$cos$(\lambda_2 z)$sinh$(\sqrt{\lambda_1^2+\lambda_2^2}y)$ & Candidate for $((\lambda_1)_1,(\lambda_2)_1)$ = $(2.1843218,2.1843218)$ & Candidate for $\sigma_1$ = 3.1019388 \\

cos$(\lambda_1 y)$cos$(\lambda_2 z)$sinh$(\sqrt{\lambda_1^2+\lambda_2^2}x)$ & & \\\hline
sinh$(\lambda_1 x)$cosh$(\lambda_2 y)$sin$(\sqrt{\lambda_1^2+\lambda_2^2}z)$ & $\lambda_1$coth$\lambda_1$ = $\lambda_2$tanh$\lambda_2$ = $\sqrt{\lambda_1^2+\lambda_2^2}$cot$(\sqrt{\lambda_1^2+\lambda_2^2})$ & $\lambda_1$coth$\lambda_1$ \\

sinh$(\lambda_1 x)$cosh$(\lambda_2 z)$sin$(\sqrt{\lambda_1^2+\lambda_2^2}y)$ & Candidate for $((\lambda_1)_1,(\lambda_2)_1)$ = $(2.9145019,2.8791763)$ & Candidate for $\sigma_1$ = 2.8974090 \\

sinh$(\lambda_1 y)$cosh$(\lambda_2 z)$sin$(\sqrt{\lambda_1^2+\lambda_2^2}x)$ & & \\
cosh$(\lambda_1 x)$sinh$(\lambda_2 y)$sin$(\sqrt{\lambda_1^2+\lambda_2^2}z)$ & & \\
cosh$(\lambda_1 x)$sinh$(\lambda_2 z)$sin$(\sqrt{\lambda_1^2+\lambda_2^2}y)$ & & \\
cosh$(\lambda_1 y)$sinh$(\lambda_2 z)$sin$(\sqrt{\lambda_1^2+\lambda_2^2}x)$ & & \\\hline
sinh$(\lambda_1 x)$sinh$(\lambda_2 y)$cos$(\sqrt{\lambda_1^2+\lambda_2^2}z)$ & $\lambda_1$coth$\lambda_1$ = $\lambda_2$coth$\lambda_2$ = $-\sqrt{\lambda_1^2+\lambda_2^2}$tan$(\sqrt{\lambda_1^2+\lambda_2^2})$ & $\lambda_1$coth$\lambda_1$ \\

sinh$(\lambda_1 x)$sinh$(\lambda_2 z)$cos$(\sqrt{\lambda_1^2+\lambda_2^2}y)$ & Candidate for $((\lambda_1)_1,(\lambda_2)_1)$ = $(1.7665698,1.7665698)$ & Candidate for $\sigma_1$ = 1.8728895 \\

sinh$(\lambda_1 y)$sinh$(\lambda_2 z)$cos$(\sqrt{\lambda_1^2+\lambda_2^2}x)$ & & \\\hline
sin$(\lambda_1 x)$cos$(\lambda_2 y)$sinh$(\sqrt{\lambda_1^2+\lambda_2^2}z)$ & $\lambda_1$cot$\lambda_1$ = $-\lambda_2$tan$\lambda_2$ = $\sqrt{\lambda_1^2+\lambda_2^2}$coth$(\sqrt{\lambda_1^2+\lambda_2^2})$ & $\sqrt{\lambda_1^2+\lambda_2^2}$coth$(\sqrt{\lambda_1^2+\lambda_2^2})$ \\

sin$(\lambda_1 x)$cos$(\lambda_2 z)$sinh$(\sqrt{\lambda_1^2+\lambda_2^2}y)$ & Candidate for $((\lambda_1)_1,(\lambda_2)_1)$ = $(2.0014790,3.8676451)$ & Candidate for $\sigma_1$ = 4.3562732 \\

sin$(\lambda_1 y)$cos$(\lambda_2 z)$sinh$(\sqrt{\lambda_1^2+\lambda_2^2}x)$ & & \\
cos$(\lambda_1 x)$sin$(\lambda_2 y)$sinh$(\sqrt{\lambda_1^2+\lambda_2^2}z)$ & & \\
cos$(\lambda_1 x)$sin$(\lambda_2 z)$sinh$(\sqrt{\lambda_1^2+\lambda_2^2}y)$ & & \\
cos$(\lambda_1 y)$sin$(\lambda_2 z)$sinh$(\sqrt{\lambda_1^2+\lambda_2^2}x)$ & & \\\hline
sin$(\lambda_1 x)$sin$(\lambda_2 y)$cosh$(\sqrt{\lambda_1^2+\lambda_2^2}z)$ & $\lambda_1$cot$\lambda_1$ = $\lambda_2$cot$\lambda_2$ = $\sqrt{\lambda_1^2+\lambda_2^2}$tanh$(\sqrt{\lambda_1^2+\lambda_2^2})$ & $\sqrt{\lambda_1^2+\lambda_2^2}$tanh$(\sqrt{\lambda_1^2+\lambda_2^2})$ \\

sin$(\lambda_1 x)$sin$(\lambda_2 z)$cosh$(\sqrt{\lambda_1^2+\lambda_2^2}y)$ & Candidate for $((\lambda_1)_1,(\lambda_2)_1)$ = $(0.7371448,0.7371448)$ & Candidate for $\sigma_1$ = 0.8119520 \\

sin$(\lambda_1 y)$sin$(\lambda_2 z)$cosh$(\sqrt{\lambda_1^2+\lambda_2^2}x)$ & & \\\hline
sinh$(\lambda_1 x)$sinh$(\lambda_2 y)$sin$(\sqrt{\lambda_1^2+\lambda_2^2}z)$ & $\lambda_1$coth$\lambda_1$ = $\lambda_2$coth$\lambda_2$ = $\sqrt{\lambda_1^2+\lambda_2^2}$cot$(\sqrt{\lambda_1^2+\lambda_2^2})$ & $\lambda_1$coth$\lambda_1$ \\

sinh$(\lambda_1 x)$sinh$(\lambda_2 z)$sin$(\sqrt{\lambda_1^2+\lambda_2^2}y)$ & Candidate for $((\lambda_1)_1,(\lambda_2)_1)$ = $(2.8949112,2.8949112)$ & Candidate for $\sigma_1$ = 2.9126739 \\

sinh$(\lambda_1 y)$sinh$(\lambda_2 z)$sin$(\sqrt{\lambda_1^2+\lambda_2^2}x)$ & & \\\hline
sin$(\lambda_1 x)$sin$(\lambda_2 y)$sinh$(\sqrt{\lambda_1^2+\lambda_2^2}z)$ & $\lambda_1$cot$\lambda_1$ = $\lambda_2$cot$\lambda_2$ = $\sqrt{\lambda_1^2+\lambda_2^2}$coth$(\sqrt{\lambda_1^2+\lambda_2^2})$ & $\sqrt{\lambda_1^2+\lambda_2^2}$coth$(\sqrt{\lambda_1^2+\lambda_2^2})$ \\

sin$(\lambda_1 x)$sin$(\lambda_2 z)$sinh$(\sqrt{\lambda_1^2+\lambda_2^2}y)$ & Candidate for $((\lambda_1)_1,(\lambda_2)_1)$ = $(3.7570495,3.7570495)$ & Candidate for $\sigma_1$ = 5.3135282 \\

sin$(\lambda_1 y)$sin$(\lambda_2 z)$sinh$(\sqrt{\lambda_1^2+\lambda_2^2}x)$ & & \\\hline

\end{tabular}
\caption{\label{tablecube}For the cube, eigenfunctions collect to give twelve categories with non-degenerate determining equations. (The other cases produce candidates of $\sigma_1 = 1$). (Calculations are rounded up to 7 decimal places.)}
\end{table}
\FloatBarrier

\vspace{1mm}
\noindent Comparing all the candidates in Table \ref{tablecube}, we see that the correct value for $\sigma_1$ is $0.5315091$, whose three eigenfunctions are cosh$(\lambda_1 x)$cosh$(\lambda_2 y)$sin$(\sqrt{\lambda_1^2+\lambda_2^2}z)$, cosh$(\lambda_1 x)$cosh$(\lambda_2 z)$sin$(\sqrt{\lambda_1^2+\lambda_2^2}y)$ and cosh$(\lambda_1 y)$cosh$(\lambda_2 z)$sin$(\sqrt{\lambda_1^2+\lambda_2^2}x)$.
\\ The surface area of our cube is 24 square units, so the first Steklov invariant for any cube is $\sigma_1 \sqrt{24}=2.603...$, which, as one can double-check, is the maximum in Figure \ref{fig:smallestoncuboid}.

\end{document}